\documentclass[twoside,leqno,10pt]{amsart}
\usepackage{amsfonts}
\usepackage{amsmath}
\usepackage{amscd}
\usepackage{amssymb}
\usepackage{amsthm}
\usepackage{amsrefs}
\usepackage{latexsym}
\usepackage{mathrsfs}
\usepackage{bbm}
\usepackage{enumerate}
\usepackage{color}
\begin{document}

\newtheorem{theorem}[subsection]{Theorem}
\newtheorem{proposition}[subsection]{Proposition}
\newtheorem{lemma}[subsection]{Lemma}
\newtheorem{corollary}[subsection]{Corollary}
\newtheorem{conjecture}[subsection]{Conjecture}
\newtheorem{prop}[subsection]{Proposition}
\numberwithin{equation}{section}
\newcommand{\mr}{\ensuremath{\mathbb R}}
\newcommand{\mc}{\ensuremath{\mathbb C}}
\newcommand{\dif}{\mathrm{d}}
\newcommand{\intz}{\mathbb{Z}}
\newcommand{\ratq}{\mathbb{Q}}
\newcommand{\natn}{\mathbb{N}}
\newcommand{\comc}{\mathbb{C}}
\newcommand{\rear}{\mathbb{R}}
\newcommand{\prip}{\mathbb{P}}
\newcommand{\uph}{\mathbb{H}}
\newcommand{\fief}{\mathbb{F}}
\newcommand{\majorarc}{\mathfrak{M}}
\newcommand{\minorarc}{\mathfrak{m}}
\newcommand{\sings}{\mathfrak{S}}
\newcommand{\fA}{\ensuremath{\mathfrak A}}
\newcommand{\mn}{\ensuremath{\mathbb N}}
\newcommand{\mq}{\ensuremath{\mathbb Q}}
\newcommand{\half}{\tfrac{1}{2}}
\newcommand{\f}{f\times \chi}
\newcommand{\summ}{\mathop{{\sum}^{\star}}}
\newcommand{\chiq}{\chi \bmod q}
\newcommand{\chidb}{\chi \bmod db}
\newcommand{\chid}{\chi \bmod d}
\newcommand{\sym}{\text{sym}^2}
\newcommand{\hhalf}{\tfrac{1}{2}}
\newcommand{\sumstar}{\sideset{}{^*}\sum}
\newcommand{\sumprime}{\sideset{}{'}\sum}
\newcommand{\sumprimeprime}{\sideset{}{''}\sum}
\newcommand{\shortmod}{\ensuremath{\negthickspace \negthickspace \negthickspace \pmod}}
\newcommand{\V}{V\left(\frac{nm}{q^2}\right)}
\newcommand{\sumi}{\mathop{{\sum}^{\dagger}}}
\newcommand{\mz}{\ensuremath{\mathbb Z}}
\newcommand{\leg}[2]{\left(\frac{#1}{#2}\right)}
\newcommand{\muK}{\mu_{\omega}}

\title[Mean values of Some Hecke characters]{Mean values of some {H}ecke characters}

\date{\today}
\author{Peng Gao and Liangyi Zhao}

\begin{abstract}
In this paper, we evaluate a smoothed character sum involving $\sum_{m}\sum_{n}\leg {m}{n}$, where $\leg{m}{n}$ is a quadratic, cubic or quartic Hecke character, and the two sums over $m$ and $n$ are of comparable lengths.
\end{abstract}

\maketitle


\section{Introduction}

 Character sums play important roles in number theory. Among various results on character sums, the following powerful large sieve type inequality for quadratic Dirichlet characters is due to D. R. Heath-Brown \cite[Theorem 1]{DRHB}:
\begin{equation}
\label{realfinal}
   \sumstar_{m \leq M} \left| \sumstar_{n \leq N} a_n \Big (\frac {n}{m} \Big ) \right|^2 \ll_{\varepsilon} (MN)^{\varepsilon}(M+N)
   \sumstar_{n \leq N} |a_n|^2,
\end{equation}
   where $(a_n)_{n\in \mathbb{N}}$ is an arbitrary sequence of complex numbers, $\varepsilon>0$, $M, N \geq 1$, the asterisks indicate that $m,n$ run over positive odd square-free
   integers and $(\frac {\cdot}{m})$ is the Jacobi symbol. \newline

The bound in \eqref{realfinal} has a variety of interesting applications, including the studies ranks of elliptic curves \cite{P&P}, mean-value and
zero-density estimates for quadratic Dirichlet $L$-functions \cites{Jutila1, Jutila, DRHB}, non-vanishing of the central value of quadratic Dirichlet $L$-functions \cite{sound1}. \newline

  As large sieve inequalities concern arbitrary sequences in general, better estimations are expected for special sequences. For example, when $a_n=1$ for all $n$,  M. V. Armon \cite[Theorem 1]{Armon} obtained the following mean square estimate for quadratic Dirichlet characters:
\begin{align}
\label{squareJacobi}
   \sum_{\substack {|D| \leq X \\ D \in \mathcal{D}}} \left| \sum_{n \leq Y} \Big (\frac {D}{n} \Big ) \right|^2 \ll XY \log X,
\end{align}
   where $\mathcal{D}$ is the set of non-square quadratic discriminants, $(\frac {D}{\cdot})$ is the Kronecker symbol. This result is better than what one gets by setting $a_n=1$ in \eqref{realfinal} and improves some earlier results of M. Jutila \cites{Jutila, Jutila1}, who obtained the same type of estimates except for higher powers of $\log X$ and applied his results to study the mean-values of class numbers of quadratic imaginary fields in \cite{Jutila1} and the second moment of Dirichlet $L$-functions with primitive quadratic characters at the central point in \cite{Jutila} .
\newline

More general even power moments of Dirichlet character sums similar are considered in \cites{Armon, Virtanen}. For corresponding estimates involving with positive odd integers, the most interesting case is the first power. Observe that in the mean square case, the left-hand side expression in \eqref{squareJacobi} involved in invariant with or without the absolute value. Therefore, one way to formulate the first moment of the quadratic Dirichlet character sum is to consider the following sum:
\begin{align} \label{CFSsum}
  S(X,Y)=\sum_{\substack {m \leq X \\ (m, 2)=1}}\sum_{\substack {n \leq Y \\ (n, 2)=1}} \leg {m}{n}.
\end{align}

   One obtains an asymptotic formula of $S(X, Y)$ for $Y=o(X/\log X)$ or $X=o(Y/\log Y)$ by a straightforward application of the P\'olya-Vinogradov inequality. Hence, the most intriguing case for evaluating $S(X, Y)$ is when $X$ and $Y$ are of comparable size. This is resolved by J. B. Conrey, D. W. Farmer and K. Soundararajan in \cite{CFS}, where an asymptotic formula of $S(X, Y)$ valid for all $X,Y$ is obtained using a Poisson summation formula developed in \cite{sound1}.

  Motivated by the above result, we study the mean values of some {H}ecke characters in this paper. More specifically, we consider the Gaussian field $K=\mq(i)$. For every element $c$ in the ring of
integers $\mathcal{O}_K = \mz[i]$ satisfying $(c, 1+i)=1$, let $(\frac {\cdot}{c}),  (\frac {\cdot}{c})_4$ be the quadratic and quartic
residue symbols defined in Section \ref{sec2.4}.  We also consider the number field $K_1=\mq(\omega)$ with $\omega=\exp(2 \pi i /3)$. For every element $c$ in the ring of integers $\mathcal{O}_{K_1} = \mz[\omega]$ satisfying $(c, 1-\omega)=1$, let $(\frac {\cdot}{c})_3$ be the cubic
residue symbol defined in Section \ref{sec2.4}. Let $\Phi(t), W(t)$ be two real-valued and non-negative smooth functions compactly supported in $(0,1)$, satisfying $\Phi(t)=W(t)=1$ for $t \in (1/U, 1-1/U)$ and such that
    $\Phi^{(j)}(t), W^{(j)}(t)\ll_j U^j$ for all integers $j \geq 0$. We define
\[  S_j(X,Y) =\sum_{n \equiv 1 \bmod {(1+i)^3}}\sum_{(m, 1+i)=1} \leg {m}{n}_j\Phi \left(\frac {N(n)}{Y} \right) W \left( \frac {N(m)}{X} \right), \quad \mbox{for} \; j=2,4, \]
and
\[  S_3(X,Y) =\sum_{n \equiv 1 \pmod {3}}\sum_{\substack {m \in \mz[\omega] \\ (m,1-\omega)=1}} \leg {m}{n}_3 \Phi\left( \frac {N(n)}{Y} \right) W \left( \frac {N(m)}{X} \right). \]
   where the sums in $S_2$, $S_4$ are over integers in $\mz[i]$ with $\leg {\cdot}{n}_2=(\frac {\cdot}{n})$ and the sums in $S_3$ are over integers in $\mz[\omega]$.   A more detailed discussion of $\leg{m}{\cdot}_j$ will be given in Section 2.  We will simply note here that $\leg{m}{\cdot}_j$ is a Hecke character modulo $16m$ for $j=2,4$ and modulo $9m$ for $j=3$. \newline

   As in the quadratic Dirichlet characters case, one expects to obtain asymptotic formulas for $S_j(X,Y),  j=2,3,4$ in a relatively easy way by applying analogues of the P\'olya-Vinogradov inequality in number fields.   With some minor changes in its proof, Lemma 2 of \cite{H&P} asserts that for $y \geq 1$ and any non-principal Hecke character $\pmod a$ of trivial infinite type,
\begin{equation}
\label{PV}
    \sum_{c \equiv 1 \bmod{ (1+i)^3}} \chi (c)\Phi \left( \frac{N(c)}{y} \right) \ll_{\varepsilon} N(a)^{(1+\varepsilon)/2}.
\end{equation}
(Lemma 2 of \cite{H&P} has a different weight function in its statement.  The proof carries over for our weight function $\Phi$ with some minor changes.)  It is easy to see that $\leg {m}{\cdot}_2$ is non-principal when $m$ is not a square. It follows from this and \eqref{PV} that
 \begin{align*}
   \sum_{n \equiv 1 \bmod {(1+i)^3}}\leg {m}{n}\Phi \left( \frac {N(n)}{Y} \right)=\begin{cases}
   \displaystyle \sum_{\substack {n \equiv 1 \bmod {(1+i)^3} \\ (n,m)=1}}\Phi \left( \frac {N(n)}{Y} \right) \qquad & \text{if $m$ is a square} ,\\
    O\left( N(m)^{(1+\varepsilon)/2} \right) \qquad & \text{otherwise}.
\end{cases}
\end{align*}
Summing over $m$, we deduce that for some constant $C$,
\begin{align*}
  S_2(X,Y) \sim CX^{1/2}Y+O(X^{3/2+\epsilon}).
\end{align*}
   An analogue expression holds when one interchanges the role of $m, n$ by the quadratic reciprocity (see Section \ref{sec2.4}). This leads to asymptotic formulas for $S_2(X,Y)$ when $X^{1+\epsilon} \ll Y$ or $Y^{1+\epsilon} \ll X$. \newline

   The above discussions apply to $S_3(X, Y)$ and $S_4(X, Y)$ as well, from which we see that just as the case of quadratic Dirichlet character sum, it is most challenging to establish asymptotic formulas for $S_j(X, Y)$, $j=2, 3, 4$ when $X$  and $Y$ are of comparable size. It is our goal in this paper to use Poisson summation for number fields to obtain asymptotic formulas for $S_j(X, Y)$, $j=2,3,4$, valid at least for certain comparable $X$, $Y$ (i.e. not covered by the formulas obtained via P\'olya-Vinogradov inequality).  Let $\zeta_{\mq(i)}(s), \zeta_{\mq(\omega)}(s)$ be the Dedekind zeta function of $\mq(i)$ and $\mq(\omega)$, respectively. We fix $U=(X/Y)^{1/2}$. Our result is
\begin{theorem}
\label{quadraticmean}
   For large $X \geq Y$, $\theta=131/416$ and any $\varepsilon>0$, we have
\begin{align*}
  S_j(X,Y) &=\frac {\pi^2 XY^{1/j}}{12\zeta_{\mq(i)}(2)}+O\left( XY^{1/j} \left(\frac {X}{Y} \right)^{-1/2}+XY^{\theta/j}+
   XY^{\varepsilon} \left( \frac {X}{Y} \right)^{1/2} \right), \quad j=2,4, \\
  S_3(X,Y) &=  \frac {\pi^2 XY^{1/3}}{27\zeta_{\mq(\omega)}(2)}+O \left( XY^{1/3} \left( \frac {X}{Y} \right)^{-1/2}+XY^{\theta/3}+XY^{\varepsilon} \left( \frac {X}{Y} \right)^{1/2} \right).
\end{align*}
\end{theorem}
Note that Theorem \ref{quadraticmean} gives a valid asymptotic formula for $S_2(X,Y)$ when $Y=o(X)$ and $X=o(Y^{2-\varepsilon})$. This goes beyond what one gets via the P\'olya-Vinogradov inequality. Similar observations apply to $S_3(X,Y)$ and $S_4(X,Y)$ as well. \newline

We conclude this section by giving a sketch of the proof of Theorem~\ref{quadraticmean} and some remarks about our result and the one in \cite{CFS}.  In proving Theorem~\ref{quadraticmean}, we start by applying the Poisson summation formula, Lemma~\ref{Poissonsumformodd}, to the sum over $m$ in $S_j(X,Y)$.  We arrive at a sum over $k \in \intz[i]$ or $k \in \intz[\omega]$ and the summands will involve Gauss sums.  The main term will come from the case in which $k=0$.  For $j=3$ or 4, the complementary sums (for $k\neq 0$) are bounded using a result of S. J. Patterson \cite{P} (see Lemma~\ref{lem1}).  The treatment of the case $j=2$ is more involved.  But in this case, we also have the explicit evaluation of the quadratic Gauss sums, Lemma~\ref{Gausssum}.  Using Mellin inversion, we are led to an expression involving Hecke $L$-functions.  Some judicious movements of the contour, together with the convexity bound for the Hecke $L$-functions and bounds derived using Lemma~\ref{Gausssum}, will enable us to arrive at the estimates needed for the case with $j=2$. It is plausible that our methods can be generalized to study mean values of quadratic Hecke characters of any imaginary quadratic number field of class number one, provided that we can have an explicit evaluation of the corresponding quadratic Gauss sum similar to the ones given in Lemma \ref{Gausssum}. As such result is not available in general, we shall be contented with our result here. \newline

We further note that, as mentioned before, the main result in \cite{CFS} gives an asymptotic formula for the sum in \eqref{CFSsum} with quadratic Dirichlet characters for all large $X$ and $Y$.  Moreover, the reader will be able to find analogues in \cite{CFS} of some of the steps in the proof in this paper.  Among the most salient differences between the results are our additional treatment of cubic and quartic Hecke characters and the use of the bounds for the Hecke $L$-functions.

\section{Preliminaries}
\label{sec 2}

\subsection{Quadratic, cubic, quartic symbols}
\label{sec2.4}
    The symbol $(\frac{\cdot}{n})_4$ is the quartic
residue symbol in the ring $\mz[i]$.  For a prime $\varpi \in \mz[i]$
with $N(\varpi) \neq 2$, the quartic symbol is defined for $a \in
\mz[i]$, $(a, \varpi)=1$ by $\leg{a}{\varpi}_4 \equiv
a^{(N(\varpi)-1)/4} \pmod{\varpi}$, with $\leg{a}{\varpi}_4 \in \{
\pm 1, \pm i \}$. When $\varpi | a$, we define
$\leg{a}{\varpi}_4 =0$.  Then the quartic character can be extended
to any composite $n$ with $(N(n), 2)=1$ multiplicatively. We extend the definition of $\leg{\cdot }{n}_4$ to $n=1$ by setting $\leg{\cdot}{1}_4=1$. We further define $(\frac{\cdot}{n})=\leg {\cdot}{n}^2_4$  to be the quadratic
residue symbol for these $n$.  \newline

The symbol $(\frac{\cdot}{n})_3$ is the cubic
residue symbol in the ring $\mz[\omega]$.  For a prime $\varpi \in \mz[\omega]$
with $N(\varpi) \neq 3$, the cubic symbol is defined for $a \in
\mz[\omega]$, $(a, \varpi)=1$ by $\leg{a}{\varpi}_3 \equiv
a^{(N(\varpi)-1)/3} \pmod{\varpi}$, with $\leg{a}{\varpi}_3 \in \{
1, \omega, \omega^2 \}$. When $\varpi | a$, we define
$\leg{a}{\varpi}_3 =0$.  Then the cubic character can be extended
to any composite $n$ with $(N(n), 3)=1$ multiplicatively. We extend the definition of $\leg{\cdot }{n}_3$ to $n=1$ by setting $\leg{\cdot}{1}_3=1$. \newline

Moreover, note that in $\intz[i]$, every ideal coprime to $2$ has a unique
generator congruent to 1 modulo $(1+i)^3$ (see the paragraph above Lemma 8.2.1 in \cite{BEW}) .  Such a generator is
called primary.  The purpose of the notion ``primary" is to eliminate the ambiguity caused by the fact that different elements of $\intz[i]$ can generate the same ideal (see Section 3, Chapter 9, \cite{I&R}).  Recall that \cite[Theorem 6.9]{Lemmermeyer} the quartic reciprocity law states
that for two primary integers  $m, n \in \mz[i]$,
\begin{equation*}
 \leg{m}{n}_4 = \leg{n}{m}_4(-1)^{((N(n)-1)/4)((N(m)-1)/4)}.
\end{equation*}

    As a consequence, the following quadratic reciprocity law holds for two primary integers  $m, n \in \mz[i]$:
\begin{align*}
 \leg{m}{n} = \leg{n}{m}.
\end{align*}

  Note also that in $\intz[\omega]$, every ideal coprime to $3$ has a unique
generator congruent to 1 modulo $3$ (see \cite[Proposition 8.1.4]{BEW}, but be aware that the notion primary in \cite[Proposition 8.1.4]{BEW} is slightly different from here).  We call such a generator primary. Recall that \cite[Theorem 7.8]{Lemmermeyer} the cubic reciprocity law states
that for two primary integers  $m, n \in \mz[\omega]$,
\begin{equation*}
 \leg{m}{n}_3 = \leg{n}{m}_3.
\end{equation*}

   Observe that an element
$n=a+bi$ of $\mz[i]$ is congruent to $1
\bmod{(1+i)^3}$ if and only if $a \equiv 1 \pmod{4}, b \equiv
0 \pmod{4}$ or $a \equiv 3 \pmod{4}, b \equiv 2 \pmod{4}$ by Lemma
6 on page 121 of \cite{I&R}. Observe also that
$n=a+b\omega \in \mz[\omega]$ is congruent to $1
\bmod{3}$ if and only if $a \equiv 1 \pmod{3}, b \equiv
0 \pmod{3}$ (see the discussions before \cite[Proposition 9.3.5]{I&R}). \newline

From the supplement theorems to the quartic and cubic reciprocity laws (see for example, \cite[Theorem 6.9]{Lemmermeyer} and \cite[Theorem 7.8]{Lemmermeyer}),
we have for $n=a+bi$ primary,
\begin{align}
\label{2.05}
  \leg {i}{n}_4=i^{(1-a)/2} \qquad \mbox{and} \qquad  \hspace{0.1in} \leg {1+i}{n}_4=i^{(a-b-1-b^2)/4},
\end{align}
and for $n=a+b\omega$ primary,
\begin{align}
\label{2.06}
  \leg {\omega}{n}_3=\omega^{(1-a-b)/3} \qquad \mbox{and} \qquad  \hspace{0.1in} \leg {1-\omega}{n}_3=i^{(a-1)/3}.
\end{align}

It is well-known that the number fields $\mq(i)$ and $\mq(\omega)$ both have class number one.  As discussed above, every ideal in $\intz[i]$ coprime to $2$ has a unique generator congruent to 1 modulo $(1+i)^3$ and every ideal in $\intz[\omega]$ coprime to $3$ has a unique
generator congruent to 1 modulo $3$. One deduces from the quadratic, quartic reciprocity and \eqref{2.05} that $\leg {m}{n}_j=1, j=2,4$ when $n \equiv 1 \pmod {16m}$. It follows from this that $\leg {m}{\cdot}_j, j=2,4$ is well-defined modulo $16m$. As $\leg {m}{\cdot}_j, j=2,4$ are trivial on units, they can be regarded as Hecke characters $\chi_j \pmod {16m}$ of trivial infinite type when $(m, 1+i)=1$.  For any ideal $(c)$ co-prime to $(1+i)$, with $c$ being the unique generator of $(c)$ satisfying $c \equiv 1 \pmod {(1+i)^3}$, $\chi_j((c))$ is defined as $\chi_j((c))=\leg {a}{c}_j$. Similarly,  one can regard $\leg {m}{\cdot}_3$ as Hecke characters $\chi_3 \pmod {9m}$ of trivial infinite type when $(m, 1-\omega)=1$.  For any ideal $(c)$ co-prime to $(1-\omega)$, with $c$ being the unique generator of $(c)$ satisfying $c \equiv 1 \pmod {3}$, $\chi_3((c))$ is defined as $\chi_3((c))=\leg {a}{c}_3$.   We refer the reader to Chapter 9 of \cite{I&R} for detailed discussions of these facts.  We shall use $\leg {m}{\cdot}_j, j=2,3,4$ to denote the corresponding Hecke characters as well, when there is no risk of confusion. In this manner, we can interpret $S_j(X,Y)$ as Hecke character sums.

\subsection{Gauss sums}
\label{section:Gauss}
For any $n \in \mz[i]$, $n \equiv 1 \pmod {(1+i)^3}$, the quadratic and quartic
 Gauss sums $g_2(n), g_4(n)$ are defined by
\[    g_2(n) =\sum_{x \bmod{n}} \leg{x}{n} \widetilde{e}_i\leg{x}{n} \qquad \mbox{and} \qquad
   g_4(n)  =\sum_{x \bmod{n}} \leg{x}{n}_4 \widetilde{e}_i\leg{x}{n}, \]
  where $ \widetilde{e}_i(z) =\exp \left( 2\pi i  \left( \frac {z}{2i} - \frac {\bar{z}}{2i} \right) \right)$.  This is analogous to the classical Gauss sums of Dirichlet characters, as all additive characters on $\intz[i]/(n)$ are of the form $\tilde{e}_i (kx/n)$ for some $k$.  Moreover, note that $g_2(1)=g_4(1)=1$ by definition. \newline

Furthermore, for any $n, r \in \mz[i]$, $n \equiv 1 \pmod {(1+i)^3}$, we set
\begin{align*}
 g_2(r,n) = \sum_{x \bmod{n}} \leg{x}{n} \widetilde{e}_i\leg{rx}{n} \qquad \mbox{and} \qquad g_4(r,n) = \sum_{x \bmod{n}} \leg{x}{n}_4 \widetilde{e}_i\leg{rx}{n}.
\end{align*}

   Similarly, for any $n \in \mz[\omega]$, $n \equiv 1 \pmod {3}$, the cubic Gauss sum $g_3(n)$ is defined by
\begin{align*}
   g_3(n) =\sum_{x \bmod{n}} \leg{x}{n}_3 \widetilde{e}_{\omega}\leg{x}{n},
\end{align*}
  where $ \widetilde{e}_{\omega}(z) =\exp \left( 2\pi i  \left( \frac {z}{\sqrt{-3}} - \frac {\bar{z}}{\sqrt{-3}} \right) \right)$.  Note that $g_3(1)=1$ by definition. \newline

  More generally, for any $n, r \in \mz[\omega]$, $n \equiv 1 \pmod {3}$, we set
\begin{align*}
 g_3(r,n) = \sum_{x \bmod{n}} \leg{x}{n}_3 \widetilde{e}_{\omega}\leg{rx}{n}.
\end{align*}

  The following property of $g_j(r,n)$ for $j=2,3,4$ can be found in \cites{Diac, B&Y} and \cite[Lemmas 2.3 and 2.4]{G&Zhao4}:
\begin{align}
\label{eq:gmult}
 g_j(rs,n) & = \overline{\leg{s}{n}}_j g_j(r,n), \quad (s,n)=1, \qquad \mbox{$n$ primary}.
\end{align}

The next lemma is an analogue of \cite[Lemma 2.3]{sound1} and allows us to evaluate $g_2(r,n)$ for $n \equiv 1 \pmod {(1+i)^3}$ explicitly.
\begin{lemma}{\cite[Lemma 2.4]{G&Zhao4}}
\label{Gausssum}
\begin{enumerate}[(i)]
\item  For $m,n$ primary and $(m , n)=1$, we have
\begin{align*}
   g_2(k,mn) & = g_2(k,m)g_2(k,n).
\end{align*}
\item Let $\varpi$ be a primary prime in $\mz[i]$. Suppose $\varpi^{h}$ is the largest power of $\varpi$ dividing $k$. (If $k = 0$ then set $h = \infty$.) Then for $l \geq 1$,
\begin{align*}
g_2(k, \varpi^l)& =\begin{cases}
    0 \qquad & \text{if} \qquad l \leq h \qquad \text{is odd},\\
    \varphi(\varpi^l)=\#(\mz[i]/(\varpi^l))^* \qquad & \text{if} \qquad l \leq h \qquad \text{is even},\\
    -N(\varpi)^{l-1} & \text{if} \qquad l= h+1 \qquad \text{is even},\\
    \leg {ik\varpi^{-h}}{\varpi}N(\varpi)^{l-1/2} \qquad & \text{if} \qquad l= h+1 \qquad \text{is odd},\\
    0, \qquad & \text{if} \qquad l \geq h+2.
\end{cases}
\end{align*}
Here $(\mz[i]/(\varpi^l))^*$ is the group of units in $(\mz[i]/(\varpi^l))$.
\end{enumerate}
\end{lemma}

\subsection{The Poisson Summation}
  The proof of Theorems \ref{quadraticmean} requires the following Poisson summation formula ( see \cite{G&Zhao4} for details):
\begin{lemma}{\cite[Corollary 2.12]{G&Zhao4}}
\label{Poissonsumformodd} Let $n \in \mz[i],   n \equiv 1 \pmod {(1+i)^3}$ and $\leg {\cdot}{n}_2$ ($\leg {\cdot}{n}_4$ ) be the quadratic (quartic) residue symbol $\pmod {n}$. For any Schwartz class function $W$,  we have
\begin{align*}
   \sum_{\substack {m \in \mz[i] \\ (m,1+i)=1}}\leg {m}{n}_j W\left(\frac {N(m)}{X}\right)=\frac {X}{2N(n)}\leg {1+i}{n}_j\sum_{k \in
   \mz[i]}(-1)^{N(k)} g_j(k,n)\widetilde{W}_i\left(\sqrt{\frac {N(k)X}{2N(n)}}\right), \quad j=2,4,
\end{align*}
   where
\begin{align*}
   \widetilde{W}_i(t) &=\int\limits^{\infty}_{-\infty}\int\limits^{\infty}_{-\infty}W(N(x+yi))\widetilde{e}_i\left(- t(x+yi)\right)\dif x \dif y, \quad t \geq 0.
\end{align*}
\end{lemma}

    Similarly, we have the following result:
\begin{lemma}
\label{Poissonsumcubic} Let $n \in \mz[\omega], n \equiv 1 \pmod {3}$ and $\leg {\cdot}{n}_3$ be the cubic symbol $\pmod {n}$. For any Schwartz class function $W$,  we have for all $a>0$,
\begin{align*}
   \sum_{m \in \mz[\omega]}\leg {m}{n}_3W\left(\frac {N(m)}{X}\right)=\frac {X}{aN(n)}\sum_{k \in
   \mz[\omega]}g_{3}(k,n)\widetilde{W}_{\omega}\left(\sqrt{\frac {N(k)X}{aN(n)}}\right),
\end{align*}
   where
\begin{align*}
   \widetilde{W}_{\omega}(t) &=\int\limits^{\infty}_{-\infty}\int\limits^{\infty}_{-\infty}W(N(x+y\omega))\widetilde{e}_{\omega}\left(- t(x+y\omega)\right)\dif x \dif y, \quad t \geq 0.
\end{align*}
\end{lemma}
\begin{proof}
   We first recall the following Poisson summation formula for
   $\mz[\omega]$ (see the proof of \cite[Lemma 10]{H})
\begin{align*}
   \sum_{j \in \mz[\omega]}f(j)=\sum_{k \in
   \mz[\omega]}\int\limits^{\infty}_{-\infty}\int\limits^{\infty}_{-\infty}f(x+y\omega)\widetilde{e}_{\omega}\left( -k(x+y\omega) \right)\dif x \dif y.
\end{align*}
   We then have
\begin{align*}
    \sum_{m \in \mz[\omega]}\leg {m}{n}_3W\left(\frac {aN(m)}{X}\right) =& \sum_{r \shortmod n}\leg {r}{n}_3\sum_{j \in \mz[\omega]}W\left(\frac {aN(r+jn)}{X}\right)  \\
   =& \sum_{r \shortmod n}\leg {r}{n}_3 \sum_{k \in
   \mz[\omega]} \ \int\limits^{\infty}_{-\infty} \int\limits^{\infty}_{-\infty}W\left(\frac {aN(r+(x+y\omega)n)}{X}\right)\widetilde{e}_{\omega}\left(-k(x+y\omega) \right) \dif
   x \dif y.
\end{align*}
   We change variables in the integral, writing
\begin{equation*}
   \sqrt{N\Big(\frac {n}{k}\Big )}\frac {k}{n}\frac {(r+(x+y\omega)n)}{\sqrt{X/a}}=u+v\omega,
\end{equation*}
   with $u,v \in \mr$. (If $k=0$, we omit the factors involving $k/n$.)  As the Jacobian of
this transformation is $aN(n)/X$, we find that
\begin{align*}
    \int\limits^{\infty}_{-\infty}\int\limits^{\infty}_{-\infty} W & \left(\frac {aN(r+(x+y\omega)n)}{X}\right)\widetilde{e}_{\omega}\left(-k(x+y\omega) \right) \dif x \dif y
   \\
   =& \frac {X}{aN(n)}\widetilde{e}_{\omega}\left(\frac {kr}{n}\right)\int\limits^{\infty}_{-\infty}\int\limits^{\infty}_{-\infty}
   W( N(u+v\omega))\widetilde{e}_{\omega}\left (-(u+v\omega)\sqrt{N(k/n)X/a} \right) \dif u \dif v,
\end{align*}
   whence
\begin{align*}
    \sum_{m \in \mz[\omega]}W\left(\frac {aN(m)}{X}\right)\leg {m}{n}_3 &= \frac {X}{aN(n)}\sum_{k \in
   \mz[\omega]}\widetilde{W}_{\omega}\left(\sqrt{\frac {N(k)X}{aN(n)}}\right)\sum_{r \shortmod n}\leg {r}{n}_3\widetilde{e}_{\omega}\left(\frac {kr}{n}\right).
\end{align*}
   As the inner sum of the last expression above is $g_3(k,n)$ by definition, this completes the proof of the lemma.
\end{proof}

   From Lemma \ref{Poissonsumcubic}, we readily deduce the following
\begin{corollary}
\label{Poissonsum3free} Let $n \in \mz[\omega], n \equiv 1 \pmod {3}$ and $\leg {\cdot}{n}_3$ be the cubic symbol $\pmod {n}$. For any Schwartz class function $W$, we have
\begin{align*}
   \sum_{\substack {m \in \mz[\omega] \\ (m,1-\omega)=1}}\leg {m}{n}_3W\left(\frac {N(m)}{X}\right)=\frac {X}{3N(n)}\leg {1-\omega}{n}_3\sum_{k \in
   \mz[\omega]}\left( \omega^{N(k)}+\overline{\omega}^{N(k)} \right) g_{3}(k,n)\widetilde{W}_{\omega}\left(\sqrt{\frac {N(k)X}{3N(n)}}\right).
\end{align*}
\end{corollary}
\begin{proof}
Note first that $1-\omega$ is a prime.  Thus it follows from Lemma \ref{Poissonsumcubic} that
\begin{equation} \label{2.21'}
\begin{split}
  \sum_{\substack {m \in \mz[\omega] \\ (m,1-\omega)=1}} \leg {m}{n}_3 & W\left( \frac {N(m)}{X} \right)=\sum_{m}\leg {m}{n}_3W \left( \frac {N(m)}{X} \right)-\leg {1-\omega}{n}_3\sum_{m} \leg {m}{n}_3W \left( \frac {3N(m)}{X} \right) \\
  &=\frac {X}{N(n)}\sum_{k \in
   \mz[\omega]}g_{3}(k,n)\widetilde{W}_{\omega}\left(\sqrt{\frac {N(k)X}{N(n)}}\right) -\leg {1-\omega}{n}_3\frac {X}{3N(n)}\sum_{k \in
   \mz[\omega]}g_{3}(k,n)\widetilde{W}_{\omega}\left(\sqrt{\frac {N(k)X}{3N(n)}}\right).
   \end{split}
\end{equation}

    Using the relation \eqref{eq:gmult}
\begin{align*}
    g_{3}((1-\omega)k,n)=\overline{\leg {1-\omega}{n}_3} g_{3}(k,n),
\end{align*}
    we can rewrite the first sum in the last expression of \eqref{2.21'} as
\begin{align*}
   \sum_{k \in
   \mz[\omega]}g_{3}(k,n)\widetilde{W}_{\omega}\left(\sqrt{\frac {N((1-\omega)k)X}{3N(n)}}\right)  &= \leg {1-\omega}{n}_3\sum_{k \in
   \mz[\omega]} g_{3}((1-\omega)k,n)\widetilde{W}_{\omega}\left(\sqrt{\frac {N((1-\omega)k)X}{3N(n)}}\right) \\
   &=\leg {1-\omega}{n}_3\sum_{\substack {k \in \mz[\omega] \\ 1-\omega| k}}g_{3}(k,n)\widetilde{W}_{\omega}\left(\sqrt{\frac {N(k)X}{3N(n)}}\right) \\
   &=\leg {1-\omega}{n}_3\sum_{k \in \mz[\omega]}\frac {1+\omega^{N(k)}+\bar{\omega}^{N(k)}}{3} g_{3}(k,n)\widetilde{W}_{\omega}\left(\sqrt{\frac {N(k)X}{3N(n)}}\right).
\end{align*}

Substituting this back to last expression in \eqref{2.21'}, we get the desired result.
\end{proof}

\subsection{Estimations of $\widetilde{W}_i(t),\widetilde{W}_{\omega}(t)$ and their derivatives}
     We will require some simple estimates on $\widetilde{W}_i(t),\widetilde{W}_{\omega}(t)$ and their derivatives. First note that for any $t \geq 0$, we have $\widetilde{W}_i(t), \widetilde{W}_{\omega}(t) \in \mr$. In
     fact, it is easy to see that
\begin{align}
\label{Wt}
     \widetilde{W}_i(t)  =\int\limits_{\mr^2}\cos (2\pi t y)W(x^2+y^2) \dif x \dif y.
\end{align}

     On the other hand, we have
\begin{align*}
     \widetilde{W}_{\omega}(t)  =\int\limits_{\mr^2}W(x^2-xy+y^2)e^{-2\pi ity} \dif x \dif y.
\end{align*}

      We change variables in the integral, writing
\begin{align*}
     x-\frac {y}{2}  =x', \quad \frac {\sqrt{3}y}{2}=y'.
\end{align*}

The Jacobian of this transformation is $\sqrt{3}/2$.  So we find that
\begin{align}
\label{Wt3}
     \widetilde{W}_{\omega}(t)  =\frac {2}{\sqrt{3}}\int\limits_{\mr^2}W(x'^2+y'^2)\cos \left( \frac {4\pi ty'}{\sqrt{3}} \right) \dif x' \dif y'.
\end{align}

    Suppose that $W(t)$ is a non-negative smooth function supported on $(0,1)$, satisfying $W(t)=1$ for $t \in (1/U, 1-1/U)$ and $W^{(j)}(t) \ll_j U^j$ for all integers $j \geq 0$.
   Using \eqref{Wt}, \eqref{Wt3}, one shows via integration by parts and our assumptions on $W$ that
\begin{align}
\label{bounds}
     \widetilde{W}^{(\mu)}_i(t),\widetilde{W}^{(\mu)}_{\omega}(t)  \ll_{j} \frac{U^j}{t^j}
\end{align}
    for all integers $\mu \geq 0, j \geq 0$ and all real $t>0$. \newline

   On the other hand, we evaluate $\widetilde{W}_i(0)$ using the polar coordinate to see that
\begin{align} \label{w0}
    \widetilde{W}_i(0) =\pi+O \left( \frac{1}{U} \right).
\end{align}
Similarly, we have
\begin{align} \label{w0'}
    \widetilde{W}_{\omega}(0) =\frac {2\pi}{\sqrt{3}}+O \left( \frac{1}{U} \right) .
\end{align}

\subsection{Analytic behavior of Dirichlet series associated with Gauss sums}
\label{section: smooth Gauss}

    In the proof of Theorem \ref{quadraticmean}, we need to know the analytic behavior of certain Dirichlet series associated with cubic or quartic Gauss sums. For this, we define
\[ G_j(s,k) =\sum_{n \equiv 1 \bmod {(1+i)^3}} \leg {1+i}{n}_j \frac {g_j(k,n)}{N(n)^{s}},  \quad \mbox{for} \; j=2,4 \quad \mbox{and} \quad G_3(s,k)=\sum_{n \equiv 1 \pmod {3}} \leg {1-\omega}{n}_3 \frac {g_3(k,n)}{N(n)^{s}}. \]
Here and in what follows, it is understood that in the first sum above $n$ runs over elements of $\mz[i]$ with $k \in \mz[i]$ and in the second sum $n$ varies over members of $\mz[\omega]$ with $k \in \mz[\omega]$. \newline

   The supplement theorems to cubic and quartic reciprocity laws \eqref{2.05}-\eqref{2.06} imply that $\leg {1-\omega}{n}_3$ is a ray class character $\pmod {9}$ in $\mq(\omega)$ and $\leg {1+i}{n}_4$ is a ray class character $\pmod {16}$ in $\mq(i)$. We shall use Lemma~\ref{Gausssum} to study $G_2(s,k)$ and deduce from a general result of S. J. Patterson \cite[Lemma, p. 200]{P} the following analytic behavior of $G_j(s,k)$, with $j=3,4$.
\begin{lemma}
\label{lem1} For $j=3,4$, the function $G_j(s,k)$ has meromorphic continuation to the half plane with $\Re (s) > 1$.  It is holomorphic in the
region $\sigma=\Re(s) > 1$ except possibly for a pole at $s = 1+1/j$. For any $\varepsilon>0$, letting $\sigma_1 = 3/2+\varepsilon$, then for $\sigma_1 \geq \sigma \geq \sigma_1-1/2$, $|s-(1+1/j)|>1/(2j)$, we have
\[ G_j(s,k) \ll N(k)^{\frac 12(\sigma_1-\sigma+\varepsilon)}(1+t^2)^{\frac {j-1}2(\sigma_1-\sigma+\varepsilon)}, \]
  where $t=\Im(s)$ and the norm is taken in the corresponding number field. Moreover, the residue satisfies
\[ \mathrm{Res}_{s=1+1/j}G_j(s,k) \ll N(k)^{a_j+\varepsilon}, \]
 where $a_3=-1/6$ and $a_4=1/4$.
\end{lemma}
We note here that Lemma~\ref{lem1} is a special case of the lemma on page 200 of \cite{P} as the latter deals with all $j \geq 1$ and also gives bounds for $G_j(s,k)$ with $\Re(s) > \sigma_1$. Though the lemma in \cite{P} is stated for $k=r^{j-2}$ with $r$ square-free, an inspection of the argument of the proof shows that it is also valid for all $k$.

\section{Proof of Theorem \ref{quadraticmean}}
\label{sec 3}

    We evaluate $S_j(X,Y)$ for $j=2,4$ first. Applying Lemma \ref{Poissonsumformodd}, we see that for $j=2,4$,
\begin{align*}
  S_j(X,Y) &=\frac {X}{2}\sum_{k \in
   \mz[i]}(-1)^{N(k)} \sum_{n \equiv 1 \bmod {(1+i)^3}} \leg {1+i}{n} \frac {g_j(k,n)}{N(n)} \Phi \left( \frac {N(n)}{Y} \right)\widetilde{W}_i\left(\sqrt{\frac {N(k)X}{2N(n)}}\right) \\
   & =\frac {X\widetilde{W}_i(0)}{2}\sum_{n \equiv 1 \bmod {(1+i)^3}} \leg {1+i}{n} \frac {g_j(0,n)}{N(n)} \Phi \left( \frac {N(n)}{Y} \right) \\
   & \hspace*{2cm} +\frac {X}{2}\sum_{\substack {k \in
   \mz[i] \\ k \neq 0}}(-1)^{N(k)} \sum_{n \equiv 1 \bmod {(1+i)^3}} \leg {1+i}{n} \frac {g_j(k,n)}{N(n)} \Phi \left( \frac {N(n)}{Y} \right) \widetilde{W}_i\left(\sqrt{\frac {N(k)X}{2N(n)}}\right) \\
   & =: M_{0,j}+M'_j.
\end{align*}

\subsection{The Term $M_{0,j}$}

    We estimate $M_{0,j}$ first. It follows straight from the definition that $g_j(0,n)=\varphi(n)$ if $n$ is a $j$-th power and $g_j(0,n)=0$ otherwise. Thus
\begin{align*}
  M_{0,j}= \frac {X\widetilde{W}_i(0)}{2}\sum_{\substack {n \equiv 1 \bmod {(1+i)^3} \\ n=\text{a $j$-th power}}}\frac {\varphi(n)}{N(n)}\Phi \left( \frac {N(n)}{Y} \right).
\end{align*}

Using the fact that if $n \equiv 1 \pmod{(1+i)^3}$, then
\begin{align*}
    \frac {\varphi(n)}{N(n)}=\sum_{\substack {d \equiv 1 \bmod {(1+i)^3} \\ d|n }}\frac {\mu_{[i]}(d)}{N(d)},
\end{align*}
    we have
\begin{align*}
   \sum_{\substack {n \equiv 1 \bmod {(1+i)^3} \\ n=\text{a $j$-th power} \\ N(n) \leq x }}\frac {\varphi(n)}{N(n)}=\sum_{\substack {d \equiv 1 \bmod {(1+i)^3} \\ N(d) \leq
   x}}\frac {\mu_{[i]}(d)}{N(d)}\sum_{\substack {n \equiv 1 \bmod {(1+i)^3} \\ n=\text{a $j$-th power}\\ N(n) \leq x/N(d)^j }}1.
\end{align*}

    Note the following result from the Gauss circle problem (we can take $\theta=131/416$, see \cite{Huxley1})
\begin{align*}
  \sum_{\substack{ n \equiv 1 \bmod {(1+i)^3} \\ N(n) \leq x}} 1 =\frac {\pi}{8}x+O \left( x^{\theta} \right),
\end{align*}
   we see that
\[  \sum_{\substack {n \equiv 1 \bmod {(1+i)^3} \\ n=\text{a $j$-th power} \\ N(n) \leq x }}\frac {\varphi(n)}{N(n)} =\frac {\pi}{8}x^{1/j}\sum_{d \equiv 1 \bmod
   {(1+i)^3}}\frac {\mu_{[i]}(d)}{N^2(d)}+O(x^{\theta/j}) =\frac {\pi}{6\zeta_{\mq(i)}(2)}x^{1/j}+O(x^{\theta/j}). \]

    Using this and partial summation, we see that
\[   M_{0,j} = \widetilde{W}_i(0)\frac {\pi XY^{1/j}}{12j\zeta_{\mq(i)}(2)}\int\limits^{\infty}_0\Phi(y)y^{1/j-1} \dif y+O \left( XY^{\theta/j} \right) = \frac {\pi^2 XY^{1/j}}{12\zeta_{\mq(i)}(2)}+O \left( XY^{\theta/j}+\frac {XY^{1/j}}{U} \right), \]
    where the last equality follows from \eqref{w0} and our assumptions on $\Phi$.

\subsection{The Term $M'_j$}
    Now suppose $k \neq 0$. By Mellin inversion, we have
\begin{align*}
    \Phi \left( \frac {N(n)}{Y} \right) \widetilde{W}_i\left(\sqrt{\frac {N(k)X}{2N(n)}}\right) = \frac 1{2\pi i}\int\limits_{(2)} \left( \frac{Y}{N(n)} \right)^s\tilde{f}(s,k) \dif s,
\end{align*}
    where
\begin{align}
\label{quadraticf}
  \tilde{f}(s,k)=\int\limits^{\infty}_{0}\Phi(t)\widetilde{W}_i\left(\sqrt{\frac {N(k)X}{2Yt}}\right) t^{s-1} \dif t.
\end{align}

     Integration by parts and using \eqref{bounds} shows $\tilde{f}(s)$ is a function satisfying the bound for all $\Re(s) > 0$, and integers $D, E>0$,
\begin{align}
\label{boundsforf}
  \tilde{f}(s,k) \ll (1+|s|)^{-D} \left( 1+\sqrt{\frac {N(k)X}{Y}} \right)^{-E+D} U^{E-1}.
\end{align}

    We have
\begin{align*}
   M'_j &= \frac {X}{2}\sum_{\substack {k \in
   \mz[i] \\ k \neq 0}}(-1)^{N(k)} \frac 1{2\pi i}\int\limits_{(2)}\tilde{f}(s,k)Y^sG_j(1+s,k) \dif s.
\end{align*}

    For the case $j=2$, we write $i(1+i)k=k_1 k^2_2$ where $k_1$ is square-free. We may write $k_1=i^a (1+i)^b k'_1$ with $a, b=0$ or $1$, $k'_1 \equiv 1 \pmod {(1+i)^3}$ and we let $\chi_{k_1}=\leg {k_1}{\cdot}$. Similar to the discussions in Section \ref{sec2.4}, we can regard $\chi_{k_1}$ as a Hecke character $\pmod {16k'_1}$ of trivial infinite type. Using this and \eqref{eq:gmult}, we have
\begin{align*}
  G_2(1+s,k) = \sum_{n \equiv 1 \bmod {(1+i)^3}} \frac {g_2\left( (1+i)k,n \right)}{N(n)^{1+s}}=L \left( \frac 12+s, \chi_{k_1} \right) \prod_{(\varpi, 1+i)=1}G_{\varpi}(s,k):=L \left( \frac 12+s, \chi_{k_1} \right) G(s,k),
\end{align*}
   where the product $\varpi$ runs over primary primes and $G_{\varpi}(s,k)$ with $(\varpi, 1+i)=1$ is defined as follows:
\begin{align*}
   G_{\varpi}(s,k)= \left( 1-\frac {\chi_{k_1}(\varpi)}{N(\varpi )^{1/2+s}} \right) \sum^{\infty}_{r=0}\frac {g \left( (1+i)k, \varpi^r \right)}{N(\varpi)^{r(1+s)}}.
\end{align*}

   By Lemma \ref{Gausssum}, we see that for a generic $\varpi \nmid (1+i)k$, $G_{\varpi}(s,k)=1-\frac 1{N(\varpi)^{1+2s}}$, this shows that $G(s,k)$ is holomorphic in
 $\Re(s) > 0$. From our evaluation of $G_{\varpi}(s,k)$ for $\varpi \nmid (1+i)k$ we see that for  $\Re{s} \geq \varepsilon$,
\begin{align*}
   G(s,k) \ll N(k)^{\varepsilon}\prod_{\varpi | k}|G_{\varpi}(s,k)|.
\end{align*}
   Suppose that $\varpi^a \| k$. By the trivial bound $|g(k, \varpi^r)| \leq N(\varpi)^r$ it follows that $|G_{\varpi}(s,k)| \leq (a+1)^2$. We then conclude that for  $\Re{s} \geq \varepsilon$,
\begin{align}
\label{boundsforG}
   G(s,k) \ll N(k)^{\varepsilon}.
\end{align}

   Using this, we see that
\begin{align}
\label{M'_2}
   M'_2 = \frac {X}{2}\sum_{\substack {k \in
   \mz[i] \\ k \neq 0}}(-1)^{N(k)} \frac 1{2\pi i}\int\limits_{(2)}\tilde{f}(s,k)Y^sL \left( \frac 12+s, \chi_{k_1} \right) G(s,k) \dif s.
\end{align}

     We now move the line of integration to the line $\Re(s) = \epsilon$. By our discussions for $j=2$ and Lemma \ref{lem1}, we see that encounter possible poles only at $s = 1/j$. Thus we may write $M'_j = M_{1,j}+R_j$, where
\[  M_{1,j} = \frac {XY^{1/j}}{2}\sum_{\substack {k \in
   \mz[i] \\ k \neq 0}} (-1)^{N(k)} \tilde{f} \left( \frac1{j}, k \right)\text{Res}_{s=1/j} G_j(1+s, k) \;\; \mbox{and} \;\;
   R_j = \frac {X}{2}\sum_{\substack {k \in
   \mz[i] \\ k \neq 0}}(-1)^{N(k)} \frac 1{2\pi i}\int\limits_{(\varepsilon)}\tilde{f}(s,k)Y^sG_j(1+s,k) \dif s.
\]

     For the case $j=2$, we use \eqref{M'_2} to see that we encounter poles only when $k_1 = 1$ (so that $L(s, \chi_{k_1}) = \zeta_{\mq[i]}(s)(1-N(1+i)^{-1/2-s})$) at $s=1/2$. Thus we may write $M'_2 = M_{1,2}+R_2$, where (by an obvious change of notation, writing $i(1+i)k^2$ in place of the corresponding $k$)
\[   M_{1,2} = \frac {XY^{1/2}}{4}\sum_{\substack {k \in
   \mz[i] \\ k \neq 0}} \tilde{f}\left( \frac1{2}, i(1+i)k^2 \right)G \left( \frac 1{2}, i(1+i)k^2 \right)\text{Res}_{s=1/2} \zeta_{\mq(i)} \left( \frac 12+s \right), \]
   and
\[  R_2 = \frac {X}{2}\sum_{\substack {k \in
   \mz[i] \\ k \neq 0}}(-1)^{N(k)} \frac 1{2\pi i}\int\limits_{(\varepsilon)}\tilde{f}(s,k)Y^s L \left( \frac 12+s, \chi_{k_1} \right) G(s,k) \dif s. \]

    Let $c_0=\pi/4$, the residue of $\zeta_{\mq(i)}(s)$ at $s=1$, and using \eqref{quadraticf}, we obtain
\begin{align*}
  M_{1,2} = \frac {c_0XY^{1/2}}{4}\int\limits^{\infty}_{0}\frac {\Phi(t)}{\sqrt{t}}\sum_{\substack {k \in
   \mz[i] \\ k \neq 0}} G \left( \frac 1{2}, i(1+i)k^2 \right) \widetilde{W}\left(\sqrt{\frac {N(k)^2X}{Yt}}\right) \dif t.
\end{align*}
We apply \eqref{bounds} with $j=2$  and \eqref{boundsforG} for all $k$ to
deduce that
\begin{align*}
  M_{1,2} \ll XY^{1/2}\frac {YU}{X}.
\end{align*}
To estimate $R_2$, we shall use the convexity bound that (see \cite[Exercise 3, p. 100]{iwakow}) for $\Re(s) \geq 0$,
\begin{align} \label{convbound}
  L \left( \frac 12+s, \chi_{k_1} \right) \ll \left( N(k_1)(1+|s|^2) \right)^{1/4+\epsilon}.
\end{align}
Using this to together with \eqref{boundsforf} by taking $D=2, E=5$ for all $k$ and \eqref{boundsforG}, we deduce that
\begin{align*}
   R_2 \ll XY^{\epsilon} \frac {Y^{3/2}U^4}{X^{3/2}} \sum_{N(k) \geq 1} \frac {N(k)^{\varepsilon}}{N(k)^{5/4}}  \ll   XY^{\epsilon} \left( \frac {Y}{X} \right)^{3/2}U^4.
\end{align*}
We thus get that
\begin{align*}
  S_2(X,Y)=\frac {\pi^2 XY^{1/2}}{12\zeta_{\mq(i)}(2)}+O \left( \frac {XY^{1/2}}{U}+XY^{\theta/2}+Y^{3/2}U
  + XY^{\epsilon} \left( \frac {Y}{X} \right)^{3/2}U^4 \right).
\end{align*}
The case $j=2$ of Theorem \ref{quadraticmean} follows by taking $U=(X/Y)^{1/2}$. \newline

  Now for $j=4$, we use Lemma~\ref{lem1} and \eqref{boundsforf} with $D=0$, $E=3$ for all $k$ and arrive at
\begin{align*}
  M_{1,4} \ll XY^{1/4}\frac {Y^{3/2}U^2}{X^{3/2}}\sum_{N(k) \geq 1} \frac {N(k)^{1/4+\varepsilon}}{N(k)^{3/2}}  \ll  XY^{1/4} \left( \frac {Y}{X} \right)^{3/2}U^2.
\end{align*}
Similarly, using Lemma \ref{lem1} and \eqref{boundsforf} with $D=2$, $E=5$ for all $k$, we deduce that
\begin{align*}
   R_4 \ll XY^{\epsilon} \frac {Y^{3/2}U^4}{X^{3/2}}\sum_{N(k) \geq 1} \frac {N(k)^{1/4+\varepsilon}}{N(k)^{3/2}} \ll  XY^{\epsilon} \left( \frac {Y}{X} \right)^{3/2}U^4.
\end{align*}
We then conclude that
\begin{align*}
  S_4(X,Y)=\frac {\pi^2 XY^{1/4}}{48\zeta_{\mq(i)}(2)}+O \left( \frac{XY^{1/4}}{U}+XY^{\theta/4}+XY^{1/4} \left( \frac {Y}{X} \right)^{3/2}U^2
  + XY^{\epsilon} \left( \frac {Y}{X} \right)^{3/2}U^4 \right).
\end{align*}
Hence, the case $j=4$ of Theorem \ref{quadraticmean} is proved by recalling that $U=(X/Y)^{1/2}$.

\subsection{Evaluation of $S_3(X,Y)$}
\label{sec 4}

    The evaluation of $S_3(X,Y)$ is similar to that of $S_4(X,Y)$. Thus, we only outline the proof here. We first use Corollary \ref{Poissonsum3free} to get
\begin{align*}
  S_3(X,Y) = M_{0,3}+M'_3,
\end{align*}
    where
\[  M_{0,3}  =\frac {2X\widetilde{W}_{\omega}(0)}{3}\sum_{n \equiv 1 \pmod {3}} \leg {1-\omega}{n}_3 \frac {g_3(0,n)}{N(n)} \Phi \left( \frac {N(n)}{Y} \right) \]
and
\[   M'_3 = \frac {X}{3}\sum_{\substack {k \in \mz[\omega] \\ k \neq 0}} \left( \omega^{N(k)}+\overline{\omega}^{N(k)} \right) \frac 1{2\pi i}\int\limits_{(2)}\tilde{f}(s,k)Y^sG_3(1+s,k) \dif s. \]

    For $M_{0,3}$, we note the following result from counting the lattice points inside an ellipse  (we can take $\theta=131/416$, see \cite{Huxley1})
\begin{align*}
  \sum_{\substack{ n \equiv 1 \pmod {3} \\ N(n) \leq x}} 1 =\frac {2\pi}{9\sqrt{3}}x+O\left( x^{\theta} \right).
\end{align*}

    Using this and partial summation, we see that
\begin{align*}
  M_{0,3} =  \frac {\pi^2 XY^{1/3}}{27\zeta_{\mq(\omega)}(2)}+O \left( XY^{\theta/3}+\frac {XY^{1/3}}{U} \right).
\end{align*}

   For $M'_3$, we move the line of integration to the line $\Re(s) = \epsilon$. We encounter possible poles only at $s = 1/3$.
We use Lemma \ref{lem1} and \eqref{boundsforf} with $D=0, E=2$ for all $k$ to estimate the residues at the poles, and we
use Lemma \ref{lem1} and \eqref{boundsforf} with $D=2, E=5$ for all $k$ to estimate the integration on the line $\Re(s) = \epsilon$ to see that
\begin{align*}
  S_3(X,Y)= \frac {\pi^2 XY^{1/3}}{27\zeta_{\mq(\omega)}(2)}+O \left( \frac {XY^{1/3}}{U}+XY^{\theta/3}+Y^{4/3}U
  +XY^{\epsilon} \left( \frac {Y}{X} \right)^{3/2}U^4 \right).
\end{align*}
This completes the evaluation for $S_3(X,Y)$ by putting in $U=(X/Y)^{1/2}$. \newline

\noindent{\bf Acknowledgments.} P. G. is supported in part by NSFC grant 11371043 and L. Z. by the FRG grant PS43707.  Parts of this work were done when P. G. visited the University of New South Wales (UNSW) in June 2017. He wishes to thank UNSW for the invitation, financial support and warm hospitality during his pleasant stay.  Finally, the authors thank the anonymous referee for his/her many comments and suggestions.

\bibliography{biblio}
\bibliographystyle{amsxport}

\vspace*{.5cm}

\noindent\begin{tabular}{p{8cm}p{8cm}}
School of Mathematics and Systems Science & School of Mathematics and Statistics \\
Beihang University & University of New South Wales \\
Beijing 100191 China & Sydney NSW 2052 Austrlia \\
Email: {\tt penggao@buaa.edu.cn} & Email: {\tt l.zhao@unsw.edu.au} \\
\end{tabular}

\end{document}